\documentclass{amsart}
\theoremstyle{plain}
\newtheorem{thm}{Theorem}

\newtheorem{lem}[thm]{Lemma}

\errorcontextlines=0 \numberwithin{equation}{section}

\begin{document}

\title[non-local curve flow]{a non-local area preserving curve flow }

\author{Liang Cheng}
\author[Li Ma]{Li Ma*}
\dedicatory{}
\date{}

\thanks{*=corresponding author. The research is partially supported by the National Natural Science
Foundation of China No.11271111 and SRFDP 20090002110019}

\address{Li Ma, Department of mathematics,
Henan Normal university Xinxiang, 453007, China}

\email{lma@math.tsinghua.edu.cn}

\address{Liang Cheng, School of Mathematics and Statistics\\
 Huazhong Normal University\\
Wuhan, 430079, P.R. CHINA} \email{math.chengliang@gmail.com}

\maketitle

\begin{abstract}
In this paper, we consider a kind of area preserving non-local flow
for convex curves in the plane. We show that the flow exists
globally, the length of evolving curve is non-increasing, and the
curve converges to a circle in $C^{\infty}$ sense as time goes into
infinity.
\end{abstract}

\section{Introduction}\label{sect1}
It is a interesting problem to study non-local flow for curves in
the plane. The purpose of this paper is to introduce a new non-local
flow which preserves the area enclosed by the evolving curve. Our
research is motivated by the famous works of Gage and Hamilton
\cite{GH} and \cite{CT} (see also \cite{CZ} for background and more
results). The curve shortening flow in a Riemannian manifold has
been studied extensively in the last few decades (see \cite{MC}).
The curve shortening flow in the plane is the family of evolving
curves $\gamma(t)$ such that
$$ \frac{\partial}{\partial
t}\gamma(t)=kN,$$ where $k$ and $N$ are the curvature of curve
$\gamma$ and the (inward pointing) unit normal vector to the curve.
For this flow, deep results are obtained in \cite{GH}, \cite{G2} and
\cite{GR}. They have proved that a simple closed initial curve
remains so along the flow, and the evolving curve becomes more and
more circular during the curve shortening process, and it converges
to a point in a finite time. Then another natural question arises
for expanding evolution flow for curves. B.Chow and D.H. Tsai have
studied the expanding flow such as $$\frac{\partial}{\partial
t}\gamma(t)=-G(\frac{1}{k})N, $$ where $G$ is a positive smooth
function with $G'>0$ everywhere. B.Andrews \cite{BA} has studied
more general expanding flows, especially flows with anisotropic
speeds. They have obtained deep results too. People then like to
study curve flow problems preserving some geometric quantities.
M.Gage \cite{G3} has considered an area-preserving flow $$
\frac{\partial}{\partial t}\gamma(t)=(k-\frac{2\pi}{L})N, $$ where
$L$ is the length of the curve $\gamma$,
 and have proved that the length of the curve is non-increasing and
 finally converges to a circle. Based on this, it is interesting study a non-local curve flow
 which preserves the length of the evolving curve. For
this, one may see \cite{MA} for a recent study. In a very recent
paper \cite{PY}, S.L.Pan and J.N.Yang consider a very interesting
length preserving curve flow for convex curves in the plane of the
form $$ \frac{\partial}{\partial
t}\gamma(t)=(\frac{L}{2\pi}-k^{-1})N,
$$
where $L$, $N$, and $k$ are the length, unit normal vector,and the
curvature  of the curve $\gamma(t)$ respectively.
 They have proved that the convex plane curve will become more and
more circular and converges to circle in the $C^{\infty}$ sense.

 We now consider the following non-local area preserving curve flow $$
\frac{\partial}{\partial t}\gamma(t)=(\alpha(t)-\frac{1}{k})N, $$
where $\alpha(t)=\frac{1}{L}\int^{L}_{0} \frac{1}{k}ds$ , and obtain
the following result.

\begin{thm}\label{main}
Suppose $\gamma(u,0)$ is a strictly convex curve (i.e. $k(0)>0$) in
the plane $\mathbb{R}^{2}.$ Assume $\gamma(t):=\gamma(u,t)$
satisfies the following evolving equation
\begin{eqnarray}\label{tf}
 \frac{\partial}{\partial t}\gamma(t)=(\alpha(t)-\frac{1}{k})N,
\end{eqnarray}
where $k$ is the curvature of the curve $\gamma(t)$ , $N$ is inward
pointing unit normal vector to the curve and
$\alpha(t):=\alpha(\gamma(t))=\frac{1}{L}\int^{L}_{0} \frac{1}{k}ds.
$
 Then the curve flow problem (\ref{tf}) has the global solution
 $\gamma(t)$, for all
 $t\in [0,\infty)$. Furthermore, the non-local curve flow(\ref{tf})
 preserves the area enclosed by the evolving curve
and keeps the strictly convexity under the evolution process. More
over, $\gamma(t)$ converges to a circle in the $C^{\infty}$ sense as
 time t goes into infinity.
\end{thm}

Note that circles are stationary solutions to (\ref{tf}). We may
assume that $\gamma(u,0)$ is not a circle. Otherwise, the result is
obvious. The interesting part in the study the non-local flow
(\ref{tf}) lies in treating the possible collapsing point where
$k=\infty$ of the evolving curve at any finite time. To overcome
this, we use the maximum principle argument. The behavior of the
curve flow is by using the isoperimetric defect property for closed
convex curves. For the local existence of the flow, we can decompose
the curvature radius function into two parts, which give a linear
PDE and nonlinear ODE. We get the local existence of the curve flow
by solving the PDE first and then solving the ODE. As a comparison,
we shall present the support function trick (\cite{JP} \cite{P}),
which has also been used by B.Chow and B.Andrews in the Gauss
curvature flow and in the curve shortening flow. We can show that
the global existence of the support functions is equivalent to the
globally existence of the non-local flow (\ref{tf}). However, this
part is new in the research of the non-local flows.  The convexity
of the evolving flow is proved by the use of maximum principle to
the curvature evolution equation. We can show that the curvature of
the evolving curve is also uniformly bounded from below by a
positive constant. Hence, using the area of the region enclosed by
the evolving convex curve is uniformly bounded, we know that the
convex region is uniformly contained in a fixed ball. By using
Bonnesen inequality (in principle, we may also use John's ellipsoid
lemma) \cite{OS}\cite{M} we know that the curvature of evolving flow
is uniformly bounded and the evolving curve becomes more round, and
then we get the global flow. We shall give full proof of this fact
in section \ref{longtimeexistence}. To prove the convergence of the
global flow, we need the argument of Gage-Hamilton \cite{GH} (see
also the works \cite{G1} \cite{G2} \cite{GL} of Gage and gage and Yi
Li). It is not clear to us how to get global existence of the
area-preserving non-local flow for curves in non-flat surfaces.

The paper is organized as follows. In section \ref{preparation}, we
 calculate some evolution equations related to this curve flow. In
section \ref{longtimeexistence}, we prove a long time existence for
the curve flow (\ref{tf}) and show the strictly convexity of the
flow is preserved. The local existence of the curve flow is done by
using the supporting function method. In section \ref{convergence},
we show that isoperimetric deficit decays to zero under the
non-local curve flow (\ref{tf}) and the evolving curve converges to
a circle in $C^{\infty}$ sense.

\section{Preparation}\label{preparation}

In this section, we calculate some formulae for more general
non-local flows than the non-local flow (\ref{tf}). Consider the
evolving curve $\gamma(t)$ defined by the map $\gamma(u,t):S^1\times
I\to \mathbb{R}^2$ satisfying the equation:
\begin{eqnarray}\label{gf}
\frac{\partial}{\partial t}\gamma(t)=(\alpha(t)-\frac{1}{k})N,
\end{eqnarray}
where $\alpha(t)$ is a $C^{\infty}$ function only depends on the
time $t$. Since $u$ and $t$ are independent variables,
 $\frac{\partial}{\partial u}$ and $\frac{\partial}{\partial
t}$ commute when applied to functions on $\mathbb{R}^2$. Let $s$
denote the arc-length of the curve $\gamma$. Then the operator
$\frac{\partial}{\partial s}$ is given in terms of $u$ by
$$
\frac{\partial}{\partial s}=\frac{1}{v}\frac{\partial}{\partial u},
$$
where $v=|\frac{\partial \gamma}{\partial u}|$.

The arc-length parameter is $ds=vdu$. Let $T$ and $N$ be the unit
tangent vector and the (inward pointing) unit normal vectors to the
curve respectively. Then the Frenet equations can be written as
$$
\frac{\partial T}{\partial u}=vkN,\ \ \frac{\partial N}{\partial
u}=-vkT.
$$

We now introduce some formulas according to (\ref{gf}). First we
have the following evolution equation for $v$.
\begin{lem} Along the flow, it holds that
 $\frac{\partial v}{\partial t}=(1-k\alpha)v$.
\end{lem}
\begin{proof}
\begin{eqnarray*}
\frac{\partial}{\partial t}(v^2)&=&\frac{\partial}{\partial
t}<\frac{\partial \gamma}{\partial u},\frac{\partial
\gamma}{\partial u}> =2<\frac{\partial \gamma}{\partial
u},\frac{\partial^2 \gamma}{\partial t\partial u}>
=2<\frac{\partial \gamma}{\partial u},\frac{\partial^2 \gamma}{\partial u\partial t}>\\
&=&2<vT,\frac{\partial}{\partial u}((\alpha-\frac{1}{k})N)>=2(1-
k\alpha)v^2.
\end{eqnarray*}

Then the lemma follows immediately.
\end{proof}

We also have the following useful relation for the operators
$\frac{\partial}{\partial t}$ and $\frac{\partial}{\partial t}$.
\begin{lem}\label{TS}Along the flow, it holds that
$$
\frac{\partial}{\partial t} \frac{\partial}{\partial
s}-\frac{\partial}{\partial s} \frac{\partial}{\partial
t}=(k\alpha-1)\frac{\partial}{\partial s}.
$$
\end{lem}
\begin{proof}
\begin{eqnarray*}
\frac{\partial}{\partial t}\frac{\partial}{\partial
s}&=&\frac{\partial}{\partial t}(\frac{1}{v}\frac{\partial}{\partial
u}) =-\frac{v_t}{v^2}\frac{\partial}
{\partial u}+\frac{1}{v}\frac{\partial}{\partial t}\frac{\partial}{\partial u}\\
&=&-\frac{v_t}{v}(\frac{1}{v}\frac{\partial} {\partial
u})+\frac{1}{v}\frac{\partial}{\partial u}\frac{\partial}{\partial
t} =(k\alpha-1)\frac{\partial}{\partial s}+\frac{\partial}{\partial
s}\frac{\partial }{\partial t}.
\end{eqnarray*}
\end{proof}

The derivatives of $T$ and $N$ are given by the following result.
\begin{lem}\label{TN} Along the flow, it holds that
$$
\frac{\partial}{\partial t}T=\frac{k_s}{k^2}N ,\ \ \textrm{and}\ \
\frac{\partial}{\partial t}N=-\frac{k_s}{k^2}T.
$$
\end{lem}
\begin{proof}
\begin{eqnarray*}
\frac{\partial}{\partial t}T&=&\frac{\partial}{\partial
t}\frac{\partial}{\partial s}\gamma=\frac{\partial}{\partial
s}\frac{\partial}{\partial
t}\gamma+(k\alpha-1)\frac{\partial}{\partial
s}\gamma\\
&=&\frac{\partial}{\partial s}((\alpha-\frac{1}{k})N)+(k\alpha-1)T\\
&=&(\frac{\partial}{\partial s}(\alpha-\frac{1}{k}))N+(\alpha-\frac{1}{k})\frac{\partial}{\partial s}N+(k\alpha-1)T\\
&=&\frac{k_s}{k^2}N.
\end{eqnarray*}
The second equation follows from
\begin{eqnarray*}
0=\frac{\partial}{\partial
t}<T,N>=<\frac{k_s}{k^2}N,N>+<T,\frac{\partial}{\partial t}N>,
\end{eqnarray*}
and $\frac{\partial N}{\partial t}$ must be perpendicular to N.
\end{proof}

We denote the angle between the tangent and the X-axis by $\theta$.
For convex curves we can use the angle $\theta$ of the tangent line
as a parameter. We may write the curvature
$k=k(\theta)=\frac{d\theta}{ds}$ in terms of this parameter. Then we
have

\begin{lem} Along the flow, it holds that
$$
\frac{\partial\theta}{\partial t}=\frac{k_s}{k^2}.
$$
\end{lem}
\begin{proof}

Since $T=(cos \theta,sin \theta)$, we use the formula in lemma
\ref{TN} to calculate
\begin{eqnarray*}
\frac{\partial T}{\partial t}=\frac{k_s}{k^2}N=\frac{k_s}{k^2}(-\sin
\theta,\cos \theta).
\end{eqnarray*}
Comparing components on both sides we get the conclusion of this
lemma.
\end{proof}

The curvature for the evolving curve evolves according to
\begin{lem}
\begin{eqnarray}\label{curveq1}
\frac{\partial k}{\partial t}= \frac{1}{k^2}\frac{\partial^2
k}{\partial s^2}-\frac{2}{k^3}(\frac{\partial k}{\partial
s})^2+(k\alpha-1)k.
\end{eqnarray}
\end{lem}
\begin{proof}
By lemma \ref{TS}, we have
\begin{eqnarray*}
\frac{\partial k}{\partial t}=\frac{\partial }{\partial
t}\frac{\partial \theta}{\partial s} &=&\frac{\partial}{\partial
s}\frac{\partial \theta}{\partial
t}+(k\alpha-1)\frac{\partial \theta}{\partial s}\\
&=&\frac{\partial }{\partial
s}(\frac{k_s}{k^2})+(k\alpha-1)k\\
&=&\frac{1}{k^2}\frac{\partial^2}{\partial
s^2}k-\frac{2}{k^3}(\frac{\partial}{\partial s}k)^2+(k\alpha-1)k.
\end{eqnarray*}
This completes the proof.
\end{proof}

Denote the area enclosed by the evolving curve by $A(t)$. Then we
have
\begin{lem}\label{area}
A(t) satisfies the equation
$$
\frac{d}{dt}A(t)=\int^L_0\frac{1}{k}ds-\alpha L.
$$
Hence, $A(t)$ remains constant provided
$\alpha(t)=\frac{1}{L}\int^L_0\frac{1}{k}ds$.
\end{lem}
\begin{proof}
Since
$$
-2A(t)=\int^L_0 <\gamma,N>ds=\int^{2\pi}_0<\gamma,vN>du,
$$
we have
\begin{eqnarray*}
-2\frac{d}{dt}A(t)&=&\int^{2\pi}_0<\gamma_t,vN>+<\gamma,v_t N>+<\gamma,vN_t>du\\
&=&\int^{2\pi}_0 <(\alpha-\frac{1}{k})N,vN>du+\int^{2\pi}_0
<\gamma,(1-k\alpha)vN>du\\
& &+\int^{2\pi}_0
<\gamma,(-\frac{k_s}{k^2})vT>du\\
&=&\int^{2\pi}_0 (\alpha-\frac{1}{k})vdu+\int^{2\pi}_0
<\gamma,(1-k\alpha)vN>du\\
& &+\int^{2\pi}_0
\frac{\partial}{\partial u}(\frac{1}{k}-\alpha)<\gamma,T>du\\
&=&\int^{L}_0 (\alpha-\frac{1}{k})ds+\int^{2\pi}_0
<\gamma,(1-k\alpha)vN>du\\
& &+\int^{2\pi}_0
\frac{\partial}{\partial u}(\frac{1}{k}-\alpha)<\gamma,T>du\\
\end{eqnarray*}
By the use of integration by parts, we have
\begin{eqnarray*}
-2\frac{d}{dt}A(t)&=&\int^{L}_0 (\alpha-\frac{1}{k})ds+\int^{2\pi}_0
<\gamma,(1-k\alpha)vN>du\\
& &+\int^{2\pi}_0 (\alpha-\frac{1}{k})(<\gamma_u,T>+<\gamma,T_u>)du\\
&=&\int^{L}_0 (\alpha-\frac{1}{k})ds+\int^{2\pi}_0
<\gamma,(1-k\alpha)vN>du\\
& &+\int^{2\pi}_0 (\alpha-\frac{1}{k})(<vT,T>+<\gamma,vkN>)du\\
&=&2\int^{L}_0(\alpha-\frac{1}{k})ds=-2(\int^L_0\frac{1}{k}ds-\alpha
L).
\end{eqnarray*}

\end{proof}

A useful lower bound for $\alpha(t)$ in the flow (\ref{tf}) is
below.
\begin{lem}\label{swartz}
If $\alpha(t)=\frac{1}{L}\int^{L}_{0} \frac{1}{k}ds$, we have
$$
\alpha\geq \frac{L}{2\pi}.
$$
The equality holds if and only if the curve $\gamma$ has the
constant curvature.
\end{lem}
\begin{proof}
Since
$$
\int^L_0 kds=2\pi,
$$
using the Cauchy-Schwartz inequality we have
$$
\int^L_0 kds \cdot \int^L_0 \frac{1}{k}ds \geq (\int^L_0ds)^2=L^2.
$$
Then we have the result.
\end{proof}

\begin{lem}\label{length}
The length of the evolving curve evolves by
$$
\frac{d}{dt}L=L-2\pi\alpha(t),
$$
Moreover, $\frac{d}{dt}L\leq 0$  provided
$\alpha(t)=\frac{1}{L}\int^L_0\frac{1}{k}ds$.
\end{lem}
\begin{proof}
\begin{eqnarray*}
\frac{d}{d t}L=\int^{2\pi}_0v_tdu=
\int^{2\pi}_0(1-k\alpha)vdu=\int^{L}_0(1-k\alpha)ds=L-2\pi\alpha.
\end{eqnarray*}

We have $\frac{d}{dt}L\leq 0$ if $\alpha(t)=\frac{1}{L}\int^{L}_{0}
\frac{1}{k}ds$ by lemma \ref{swartz}.
\end{proof}

So much for the general flow (\ref{gf}).

\section{Local and long time existence}\label{longtimeexistence}

We first consider a priori estimates of the curve flow. Since the
changing of the tangential components of the velocity vector of
$\gamma_t$ affects only the parametrization, not the geometric
shapes of the evolving curve, we can choose a suitable tangent
component $\eta$ to simplify the analysis of the non-local flow
(\ref{tf}). This trick has been used by many authors, see, for
example, \cite{GH} or \cite{PY}. So we consider the following
evolution problem, which is equivalent to (\ref{tf}):
\begin{eqnarray}\label{change}
\gamma_t=(\alpha(t)-\frac{1}{k})N+\eta T.
\end{eqnarray}
Similar to the calculations in section \ref{preparation}, we have
\begin{lem}\label{change1} Along the flow (\ref{change}), it is true that
\begin{eqnarray*}
&&\frac{\partial v}{\partial t}=\frac{\partial \eta}{\partial u}+(1-k\alpha)v,\\
 &&\frac{\partial}{\partial t}T=(\eta k+\frac{k_s}{k^2})N ,\ \
\frac{\partial}{\partial t}N=-(\eta k+\frac{k_s}{k^2})T, \\
&&\frac{\partial \theta}{\partial t}=\eta k+\frac{k_s}{k^2},\\
&&\frac{\partial k}{\partial t}= \frac{1}{k^2}\frac{\partial^2
k}{\partial s^2}-\frac{2}{k^3}(\frac{\partial k}{\partial
s})^2+(k\alpha-1)k+\eta \frac{\partial k}{\partial s}\\
&&\frac{d}{dt}A(t)=\int^L_0\frac{1}{k}ds-\alpha L,\\
&&\frac{d}{dt}L=L-2\pi\alpha.\\
\end{eqnarray*}
\end{lem}

Note that $L$ and $A$ are both independent of $\eta$. In order to
make $\theta$ independent of time $t$, we can choose suitable $\eta$
such that $\frac{\partial \theta}{\partial t}=0$, i.e.
$$\eta=-\frac{1}{k^3}k_s=-\frac{1}{k^2}\frac{\partial k}{\partial
\theta}.$$ Then by changing the space variable we can transform away
the tangential component, without changing the shape of the curves
(see also the proof of Theorem 4.1.4 in \cite{GH}). Hence we can get
from the flow (\ref{change}) the flow (\ref{tf}).

We now consider the following equivalent problem instead from now
on:
\begin{eqnarray}\label{change3}
\gamma_t=(\alpha(t)-\frac{1}{k})N-\frac{1}{k^3}\frac{\partial
k}{\partial s}T.
\end{eqnarray}

Then by lemma \ref{change1}, we have the following result.
\begin{lem}\label{change2} Along the flow (\ref{change3}), it holds that
\begin{eqnarray}
 &&\frac{\partial}{\partial t}T=0 ,\ \
\frac{\partial}{\partial t}N=0, \ \ \frac{\partial \theta}{\partial
t}=0,\\
 &&\frac{\partial k}{\partial t}=
\frac{1}{k^2}\frac{\partial^2 k}{\partial
s^2}-\frac{3}{k^3}(\frac{\partial k}{\partial s})^2+(k\alpha-1)k,\label{curveq5}\\
&&\frac{d}{dt}A(t)=\int^L_0\frac{1}{k}ds-\alpha L,\\
&&\frac{d}{dt}L=L-2\pi\alpha.
\end{eqnarray}
\end{lem}

By theorem \ref{convex}, we can use the angle variable $\theta$ of
the tangent line as a parameter for convex curves. To determine the
evolution equation for curvature of the evolving curve when using
$\theta$ as a parameter, we take $\tau=t$
 as the time parameter. That is, we change variables from
 $(u,t)$ to $(\theta,\tau)$. We obtain the following equation
 for $k$ in terms of $\theta$ and $\tau$.

\begin{lem}
\begin{eqnarray}\label{curveq2}
\frac{\partial k}{\partial \tau}=\frac{\partial ^2 k}{\partial
\theta^2} -\frac{2}{k}(\frac{\partial k}{\partial
\theta})^2+(k\alpha-1)k.
\end{eqnarray}
\end{lem}
\begin{proof}
By the chain rule and lemma \ref{change2}, we have
$$
\frac{\partial k}{\partial t}=\frac{\partial k}{\partial \tau}+
\frac{\partial k}{\partial \theta}\frac{\partial \theta}{\partial t}
=\frac{\partial k}{\partial \tau},
$$
and
$$
\frac{\partial^2 k}{\partial s^2}=(\frac{\partial \theta}{\partial
s}\frac{\partial }{\partial \theta}) (\frac{\partial
\theta}{\partial s}\frac{\partial }{\partial \theta})
=k^2\frac{\partial^2 k}{\partial \theta^2}+k(\frac{\partial
k}{\partial \theta})^2.
$$
Substituting these expressions into the formula (\ref{curveq5}) in
lemma \ref{change2} we get the result.
\end{proof}

Note that $L=\int_{S^1}\frac{1}{k}d\theta$. By direct calculation,
we can derive a heat equation for $1/k$ (see (\ref{curveq3})) from
formula (\ref{curveq2}).

\begin{lem}\label{longtime} We have
\begin{eqnarray}\label{curveq4}
\frac{\partial }{\partial \tau}(\frac{1}{k})=\frac{\partial ^2
}{\partial \theta^2}(\frac{1}{k}) +\frac{1}{k}-\alpha.
\end{eqnarray}
Let $h=\frac{1}{k}-\frac{L}{2\pi}$ and let $w=he^{-\tau}$. Then we
have
$$
h_{\tau}=h_{\theta\theta}+h
$$
and
\begin{eqnarray}\label{curveq3}
w_{\tau}=w_{\theta\theta}.
\end{eqnarray}
Then $w$ can be solved for time interval $[0,+\infty)$ as
$$
w(\theta,\tau)=\int^{\infty}_{-\infty}\frac{1}{2\sqrt{\pi
\tau}}e^{-\frac{(\theta-\xi)^2}{4\tau}}w(\theta,0)d\xi
$$
and the solution to the flow (\ref{tf}) is smooth.
\end{lem}
\begin{proof}
Since
$$
\frac{\partial }{\partial \tau}(\frac{1}{k})=-\frac{k_{\tau}}{k^2}=
-\frac{k_{\theta\theta}}{k^2}+\frac{2}{k^3}k^2_{\theta}-\frac{k^2\alpha-k}{k^2},
$$
and
$$
\frac{\partial ^2 }{\partial \theta^2}(\frac{1}{k})=-\frac{\partial
}{\partial \theta}(\frac{k_{\theta}}{k^2})
=-\frac{k_{\theta\theta}}{k^2}+\frac{2}{k^3}k^2_{\theta},
$$
(\ref{curveq4}) follows immediately.

By lemma \ref{length}, we have
\begin{eqnarray*}
h_{\tau}&=&\frac{\partial }{\partial
\tau}(\frac{1}{k})-\frac{L_{\tau}}{2\pi}=
\frac{\partial ^2 }{\partial \theta^2}(\frac{1}{k})+\frac{1}{k}-\alpha-\frac{L-2\pi \alpha}{2\pi}\\
&=&\frac{\partial ^2 }{\partial \theta^2}(\frac{1}{k})+(\frac{1}{k}-\frac{L}{2\pi})\\
&=&h_{\theta\theta}+h.
\end{eqnarray*}
Then (\ref{curveq3}) follows immediately.
\end{proof}
By lemma \ref{longtime} we know that the function $h$ is globally
well-defined from the initial data $h(0)$ of the curve $\gamma(0)$.
Note that $\int_{S^1}hd\theta=0$. Using
$\frac{1}{k}=h+\frac{L}{2\pi}$ and $\alpha=\frac{1}{L}\int_0^L
\frac{1}{k}ds$, we know that
$$
\alpha=\frac{1}{L}\int_0^{2\pi}
\frac{1}{k^2}d\theta=\frac{1}{L}\int_{S^1}
h^2d\theta+\frac{L}{2\pi}.
$$
Then from the ODE
$$
\frac{d}{d\tau}L=L-2\pi\alpha=-\frac{1}{L}\int_{S^1} h^2d\theta,
$$
we can solve $L$ from the initial curve $\gamma(0)$ and then we get
$\frac{1}{k}$ in local time interval. Then, we can get the local
existence of the curve flow via the formula $$ x(\theta, t)=
\int_0^\theta \frac{\cos \phi}{k} d\phi, \ \  y(\theta, t)=
\int_0^\theta \frac{\sin \phi}{k} d\phi
$$
to define evolving curves $\gamma(t)$ (as in \cite{GH}) for the flow
equation (\ref{change3}). As a comparison, we shall try to consider
the local existence by using the supporting function method since it
is often used in the literatures about curve evolution flows.

\begin{thm}\label{convex}
Under the assumptions of theorem \ref{main},
 the curve flow keeps the convexity property.
\end{thm}

\begin{proof}By lemma \ref{longtime}, there exists a constant $M>0$ such that for
$(\theta,\tau)\in [0,2\pi]\times (0,\infty)$,
$$
|w(\theta,\tau)|\leq M.
$$

Then for any finite $T^{*}$, $\tau\in [0,T^*)$,
$$
|\frac{1}{k}-\frac{L}{2\pi}|\leq M e^{T^*}.
$$

By lemma \ref{length}, L is bounded above. Also by the isoperimetric
inequality, L has a lower bound $\sqrt{4\pi A}$. So we get
$$k(\theta,\tau)\neq 0, \ \text{for} \ (\theta,\tau)\in [0,2\pi]\times [0,T^*).$$
Now, from the continuity of $k(\theta,\tau)$ and the positivity of
$k(\theta,0)$, we know that
$$
k(\theta,\tau)> 0, \ \text{for} \ (\theta,\tau)\in [0,2\pi]\times
[0,T^*).
$$

Then the theorem follows from the arbitrariness of $T^*$.
\end{proof}

One can also see that $\frac{1}{k}$ is uniformly bounded at any
finite existing time interval $[0,T)$. This then implies that
$\alpha(\tau)\leq C(T)$ for some constant $C(T)>0$. In fact, by the
maximum principle we know that $\inf_{S^1}w(\theta,\tau)$ is
non-decreasing and $\sup_{S^1}w(\theta,\tau)$ is non-increasing.
This implies that $$\frac{1}{k}-\frac{L}{2\pi}\leq
[\frac{1}{k(0)}-\frac{L(0)}{2\pi}]e^{\tau}.$$ Then
$$
\frac{1}{k}\leq\frac{L}{2\pi}+\sup_{S^1}[\frac{1}{k(0)}-\frac{L(0)}{2\pi}]e^{\tau},
$$
which gives a lower bound of $k$ in any time interval. Here we have
used the fact that $\int_{\gamma}k(0)ds=2\pi$, which gives
$L(0)\inf_{S^1} k(0)<2\pi$ (and $\sup_{S^1}k(0)^{-1}=(\inf
k(0))^{-1}>\frac{L(0)}{2\pi}$) unless $\gamma(0)$ is the circle.
Similarly, we have
$$
\frac{1}{k}\geq\frac{L}{2\pi}+\inf_{S^1}[\frac{1}{k(0)}-\frac{L(0)}{2\pi}]e^{\tau}.
$$
More importantly, these two estimates imply that there is no blow-up
of the evolving curve in any finite time interval.

By this, we have proved the existence of global flow of (\ref{gf}).
\begin{thm}\label{no+blowup+1} Assume the local existence of the curve flow (\ref{gf}).  We have a global flow to the curve flow
(\ref{gf}), that is, there is no finite time blow up point of the
curvature function $k$.
\end{thm}

Now we present the supporting function method to prove the local
existence of the curve flow (\ref{gf})(see theorem I1.2 in
\cite{BA}). We denote $S$ the support function of the curve
$\gamma$, i.e. $S=-<\gamma,N>$. So, $L=\int_{S^1} Sd\theta$ and
\begin{eqnarray}\label{suppcurva}
\frac{1}{k}=\frac{\partial^2 S}{\partial \theta^2}+S.
\end{eqnarray}
Since $\alpha=\frac{1}{L}\int_0^L\frac{1}{k}ds$, we then have
$$
\alpha=\frac{1}{L}\int_{S^1}(\partial^2_{\theta} S+S)^2d\theta.
$$

We have the following evolution equation of support function.

\begin{lem}
\begin{eqnarray}\label{support2}
\frac{\partial S}{\partial \tau}=\frac{\partial^2 S}{\partial
\theta^2}+S-\alpha
\end{eqnarray}
\end{lem}
\begin{proof}
By lemma \ref{change2}, we have
\begin{eqnarray*}
\frac{\partial S}{\partial \tau}&=&-\frac{\partial }{\partial
\tau}<\gamma,N>=-<\frac{\partial }{\partial \tau}\gamma,N>\\
&=&-<(\alpha(\tau)-\frac{1}{k})N-\frac{1}{k^2}\frac{\partial
k}{\partial \theta}T,N>\\
&=&\frac{1}{k}-\alpha\\
&=&\frac{\partial^2 S}{\partial \theta^2}+S-\alpha.
\end{eqnarray*}
\end{proof}

Then we have
$$
[(S-\frac{L}{2\pi})e^{-\tau}]_\tau=[(S-\frac{L}{2\pi})e^{-\tau}]_{\theta\theta}.
$$

Similar to lemma \ref{longtime}, we have
\begin{thm}\label{support}
 The support function $S$ can be solved for time interval
$[0,+\infty)$ as
$$
(S(\theta,\tau)-\frac{L(\tau)}{2\pi})e^{-\tau}=\int^{\infty}_{-\infty}\frac{1}{2\sqrt{\pi
\tau}}e^{-\frac{(\theta-\xi)^2}{4\tau}}(S(\theta,0)-\frac{L(0)}{2\pi})d\xi.
$$
Furthermore,
$$
(\frac{1}{k}-S(\theta,\tau))e^{-\tau}=\int^{\infty}_{-\infty}\frac{1}{2\sqrt{\pi
\tau}}e^{-\frac{(\theta-\xi)^2}{4\tau}}(\frac{1}{k}(0)-S(\theta,0))d\xi.
$$
\end{thm}

From all these, we can easily get the following.
\begin{thm}\label{no+blowup}
For $\tau>0$, we have
$$
(\frac{1}{k}-S)(\theta,\tau)\geq e^\tau\inf_\theta
(\frac{1}{k}(0)-S(\theta,0)).
$$
\end{thm}
With these understanding, we can use the general existence result of
Jiang- Pan \cite{JP} (or the method used in \cite{MA}) to show that
there is a local solution to the flow (\ref{support2}).
 It is
convenient to choose the normal vector for parameter of the curve.
We denote $\textbf{n}:\gamma \to S^1$ be the Gauss map. Let $z$ be
the normal vector and $\gamma$ parametrized by $z$. So we have
$S(z,t)=-<\gamma(n^{-1}(z)),z>$. Let $r[S](z)$ be the radius of
curvature at the point with normal $z$ is given by
$r[S](z)=\frac{\partial^2 S}{\partial \theta^2}(z)+S(z)$. Then we
obtain the following result.

\begin{thm}\label{support1}
Assume $S:S^1\times [0,\infty)\to \mathbb{R}$ is a smooth function
of equation (\ref{support2}) with radius $r[S]>0$, then there exists
a solution $\gamma:\zeta\times [0,\infty)\to \mathbb{R}$ satisfies
the equation (\ref{tf}) which has the initial data
$\gamma_0=\{-S_0(z)z-\frac{\partial S_0}{\partial \theta}(z)
\frac{\partial z}{\partial \theta}:z\in S^1 \}$ and such that the
curve $\gamma(t)$ has the support function $S(t)$ for each
$t\in[0,\infty)$.
\end{thm}
\begin{proof}
We define the evolving curve $\bar{\gamma}:S^1\times [0,\infty)\to
\mathbb{R}$ by
\begin{eqnarray}\label{defsurpport}
\bar{\gamma}(z,t)=-S(z,t)z-\frac{\partial S}{\partial
\theta}(z,t)\frac{\partial z}{\partial \theta}.
\end{eqnarray}
Then we have an evolving curve $\gamma(t)$ which has the support
function $S$ and the curvature $k$ satisfying (\ref{suppcurva}). By
the assumptions of this theorem,
\begin{eqnarray*}
\frac{\partial \bar{\gamma}}{\partial t}(z,t) &=&-\frac{\partial
S}{\partial t}(z,t) z-\frac{\partial^2 S}{\partial t\partial
\theta}\frac{\partial z}{\partial \theta}\\
&=&-(\frac{\partial^2 S}{\partial
\theta^2}+S-\alpha)z-\frac{\partial }{\partial
\theta}(\frac{\partial^2 S}{\partial
\theta^2}+S-\alpha)\frac{\partial z}{\partial \theta}\\
&=&(\alpha-r[S](z))z-\frac{\partial }{\partial
\theta}(r[S](z))\frac{\partial z}{\partial \theta}\\
&=&(\alpha-\frac{1}{k_{\bar{\gamma}}})N_{\bar{\gamma}}(z)-T\bar{\gamma}(V),
\end{eqnarray*}
where $N_{\bar{\gamma}}$ and $k_{\bar{\gamma}}$ are the normal and
curvature corresponding to $\bar{\gamma}$, and $V\in
TS^1\times[0,\infty)$ is the vector field on $S^1$ given by
$k_{\bar{\gamma}}\frac{\partial }{\partial
\theta}(\frac{1}{k_{\bar{\gamma}}})\frac{\partial z}{\partial
\theta}$. Here we used the fact
$T\bar{\gamma}(V)=k_{\bar{\gamma}}^{-1}V$ for any $V\in TS^1$. Next
we define a family of diffeomorphisms $\phi$ such that
$\gamma(p,t)=\bar{\gamma}(\phi(p,t),t)$ gives the solution of
equation (\ref{tf}). Now we take $\phi(p,t)$ to solve the following
ordinary differential equation for each $p$ :
\begin{eqnarray}
\frac{d}{dt}\phi(p,t)=V(\phi(p,t),t).
\end{eqnarray}
This equation has a unique solution for each $p$ as long as $S$
exists and remains smooth. Then we have
\begin{eqnarray*}
\frac{\partial}{\partial t}\gamma(p,t)&=&\frac{\partial}{\partial
t}\bar{\gamma}(\phi(p,t),t)\\
&=&(\frac{\partial}{\partial
t}\bar{\gamma})(\phi(p,t),t)+T\bar{\gamma}(\frac{\partial}{\partial
t}\phi(p,t),t)\\
&=&(\alpha-\frac{1}{k_{\bar{\gamma}}(\phi(p,t),t)})N_{\bar{\gamma}}(\phi(p,t),t)-T\bar{\gamma}(V)+T\bar{\gamma}(V)\\
&=&(\alpha-\frac{1}{k_{\gamma}(p,t)}) N_{\gamma}(p,t),\\
\end{eqnarray*}
where we have used $k_{\gamma}(p,t)=k_{\bar{\gamma}}(\phi(p,t),t)$
and $N_{\gamma}(p,t)=N_{\bar{\gamma}}(\phi(p,t),t)$. Hence the
theorem holds.
\end{proof}

The above result implies the local existence of the curve flow.
Hence we get the following result immediately by the use of theorem
\ref{support} and theorem \ref{support1}.
\begin{thm}
Under the assumptions of theorem \ref{main}, the curve flow
(\ref{tf}) has the global solution, that the flow exists in time
interval
 $[0,\infty)$ with initial curve $\gamma(0)$.
\end{thm}

\section{Convergence}\label{convergence}

In this section we prove the convergence of the evolving curves.

In order to understand the behavior of the global curve flow, we
need the following isoperimetric inequality due to S.L.Pan and
J.N.Yang.
\begin{thm}\cite{PY}\label{isoperi}
For the closed, convex $C^{2}$ curves in the plane, we have
$$
\frac{L^2-2\pi A}{\pi} \leq \int_{0}^{L} \frac{1}{k}ds,
$$
where $L,A$ and $k$ are the length of the curve, the area enclosed
by the evolving curve, and its curvature.
\end{thm}
Recall the following Bonnesen inequality (\cite{OS}) that
\begin{equation}\label{osserman}
L^2-4\pi A\geq A^2(\frac{1}{r_{in}}-\frac{1}{r_{ou}})^2,
\end{equation}
where $r_{in}$ and $r_{ou}$ are radii of the incircle (the largest
circle contained in the domain enclosed by $\gamma$) and the
circumcircle ( the smallest circle containing $\gamma$). Since along
the flow, the area is fixed and then the curve $\gamma$ becomes more
round provided the isoperimetric deficit $L^{2}-4\pi A$ is
non-increasing.

In fact, we have the following result, which shows that the curve
flow becomes more and more circular under the evolution process.
\begin{thm}\label{isodeficit}
If a convex curve evolves according to (\ref{gf}), then the
isoperimetric deficit $L^{2}-4\pi A$ is non-increasing during the
evolution process and in case of global flow, it converges to zero
as the time $\tau$ goes to infinity.
\end{thm}

\begin{proof}
By lemma \ref{area} and lemma \ref{length}, we have
\begin{eqnarray*}
\frac{d}{d\tau}(L^{2}-4\pi A)&=&2LL_t-4\pi A_\tau\\
&=&2L(L-2\pi\alpha)-4\pi(\int^L_0\frac{1}{k}ds-\alpha L)\\
&=&2L^2-4\pi\int^L_0\frac{1}{k}ds.
\end{eqnarray*}
By the theorem \ref{isoperi}, we have
$$
\frac{d}{dt}(L^{2}-4\pi A)\leq 2L^2-4\pi \frac{L^2-2\pi A}{\pi} \leq
-2(L^2-4\pi A),
$$
We always have $L^2-4\pi A\geq 0$, and so
$$
\frac{d}{d\tau}(L^{2}-4\pi A)\leq 0.
$$
Moreover, we have
$$
0\leq L^{2}-4\pi A\leq Cexp(-2\tau),
$$
where $C=L^{2}(0)-4\pi A(0)$. As $t\rightarrow \infty$ in case of
global flow, we have the decay of the isoperimetric defect,
$$
L^{2}-4\pi A\rightarrow 0.
$$
\end{proof}

Recall that the area is fixed along our curve flow. By
(\ref{osserman}), we know that the isoperimetric defect for closed
convex curve is the measure of circularness of the curve (see
\cite{OS} and \cite{SC} for more related inequalities). Then we know
that for $\tau>0$, $\gamma(\tau)$ is more circular than $\gamma(0)$
but with fixed area of their enclosed regions.

We now give a remark for higher derivative bounds of $k$. We may
also get higher order derivatives estimates for $1/k$. Let
$(\frac{1}{k})_\theta=\partial_\theta(\frac{1}{k})$. Then we have
$$
\frac{\partial }{\partial \tau}(\frac{1}{k})_\theta=\frac{\partial
^2 }{\partial \theta^2}(\frac{1}{k})_\theta +(\frac{1}{k})_\theta,
$$
or
$$
\frac{\partial }{\partial
\tau}[e^{-\tau}(\frac{1}{k})_\theta]=\frac{\partial ^2 }{\partial
\theta^2}[e^{-\tau}(\frac{1}{k})_\theta].
$$
 Using the maximum principle, we know that $(\frac{1}{k})_\theta$
is bounded at any existing time. In fact we have
$$
(\frac{1}{k})_\theta
e^{-\tau}=\int^{\infty}_{-\infty}\frac{1}{2\sqrt{\pi
\tau}}e^{-\frac{(\theta-\xi)^2}{4\tau}}(\frac{1}{k})_\theta(0)d\xi.
$$

 We
denote $r_{in}$ radii of the largest inscribed circle of the curve
$\gamma$. Now we can use the method of M.Gage and R.S.Hamilton to
show curvature $k$ converging to a constant as time goes into
infinity, see Section 5 in \cite{GH}. First, we need a result in
\cite{GH}.

\begin{thm}\cite{GH}\label{GH}
$k(\theta,t)r_{in}(t)$ converges uniformly to 1, when the
isoperimetric deficit $L^{2}-4\pi A \to 0$.
\end{thm}

\begin{thm}\label{constant}
Under the assumptions of theorem \ref{main}, we have $k\to
\frac{2\pi}{L}$ as $t\to \infty$.
\end{thm}
\begin{proof}
By the Bonnesen inequality (see \cite{OS}),
\begin{equation}\label{bonn}
\frac{L^2}{A}-4\pi \geq \frac{(L-2\pi r_{in})^2}{A}
\end{equation}
 and theorem \ref{isodeficit}, theorem \ref{area},
 we have $r_{in}\to \frac{L}{2\pi}$ as $t\to\infty$. Hence the
 theorem follows immediately from theorem \ref{GH}.
\end{proof}

Then we obtain the $C^{\infty}$ convergent part in theorem
\ref{main}.

\begin{thm}\label{cconvegence}
Under the assumptions of theorem \ref{main}, the curve flow
(\ref{tf}) converges to a circle in $C^{\infty}$ sense as time goes
into infinity.
\end{thm}
\begin{proof}
By lemma \ref{longtime}, the curvature $k(t)$ is $C^{\infty}$
differentiable. Then theorem \ref{main} follows immediately from
theorem \ref{constant}.
\end{proof}


\begin{thebibliography}{99}
\bibitem{BA}B.Andrews, \emph{Evolving convex curves}. Calc.Var.PDE's,
7,315-371(1998).

\bibitem{B} T.Bonnesen, W.Fenchel, Theorie der Convexen
K$\ddot{o}$rper.Chelsea,New York, 1948.

\bibitem{CZ} K.S.Chou, X.P.Zhu, The curve shortening Problem. CRC
Pre ss.Boca Raton, 2001.

\bibitem{CT}B.Chow, D.H.Tsai,\emph{ Geometric expansion of convex plane curves}.
J.Diff.Geom.,44,312-330(1996)

\bibitem{BC} B.Chow, P.Lu,L.Ni, Hamilton's Ricci Flow. Science
Press/American Mathematical Society, Beijing/Providence,(2006).

\bibitem{G1} M.Gage, \emph{An isoperimetric inequality with applications to curve shortening}. Duke Math.J.,50,1225-1229(1983).

\bibitem{G2} M.Gage, \emph{Curve shortening makes convex curves circular}. Invent.Math., 76(1984) 357-364

\bibitem{G3} M.Gage, \emph{On an area-preserving evolution equation for plane curves}. In:DeTurck,D.M.(ed,)
Nonlinear Problems in Geometry,Contemp.Math,vol.51,pp.51-62(1986)

\bibitem{GH} M.Gage, R.S.Hamilton, \emph{The heat equation shrinking convex
plane curves}. J.Diff.Geom., 23,69-96(1986).

\bibitem{GL} M.Gage, Yi Li, \emph{Evolving plane curves by curvature in
relative geometries}. II. Duke Math. J. 75 (1994), no. 1, 79-98.

\bibitem{GO} M.Green, S.Osher, \emph{Steiner polynomials, Wulff flows, and some new isoperimetric
inequalities for convex plane curves}. Asian J. Math. 3, 659-676
(1999)

\bibitem{GR} M.Grayson, \emph{The heat equation shrinks embeded plane
curves to round points}. J.Diff.Geom.26,285-314(1987).

\bibitem{H1} G.Huisken, \emph{Flow by mean curvature of convex surfaces
into spheres}. J.Diff.Geom.,20,237-266(1984).

\bibitem{H2} G.Huisken, \emph{The volume preserving mean curvature flow}. J. Reine Angew.
Math., 382, 35-48 (1987)

\bibitem{JP}L.S.Jiang, S.L.Pan, \emph{On a non-local curve evolution problem in
the plane}. Commun. Anal. Geom. 16, 1-26 (2008)

\bibitem{MC} L.Ma, D.Z.Chen, \emph{Curve shortening in a Riemannian
manifold}. Ann.Mat.Pura.Appl., 186,663-684(2007).

\bibitem{MA} L.Ma, A.Q.Zhu, \emph{On a length preserving curve flow}.
Monatsh Math (2012) 165:57-78

\bibitem{M}
De G. Miguel, \emph{Differentiation of integrals in Rn}; Lectures in
Math. 481,(1977, Springer-Verlag.

\bibitem{OS} R.Osserman, \emph{Bonnesen isoperimetric inequalities}.
Amer.Math.Monthly, 86(No.1)(1979)1-29.

\bibitem{SC} R.Schneider,\emph{ Convex bodies: The Brunn-Minkowski theory}. Encyclopedia of Mathematics
and its Applications Vol. 44, Cambridge University Press, (1993).

\bibitem{P3} Grisha Perelman, \emph{Finite time extinction for the solutions to the Ricci flow on certain three-manifold}.
  math.DG/0307245, 2003.

\bibitem{P} S.L. Pan,\emph{ A note on the general curve flows}. J. Math. Stud. 33, 17-26 (2000)

\bibitem{PY} S.L. Pan, J.N. Yang, \emph{On a non-local perimeter-preserving curve evolution problem for convex plane curves}.
manuscripta math.,127,(2008)469-484.

\bibitem{Z} X.P.Zhu, \emph{Lectures on mean curvature flows}. AMS/IP Studies in Advanced
Mathematics, vol. 32. American Mathematical Society/International
Press, Providence/Somerville (2002)

\end{thebibliography}
\end{document}